\documentclass[11pt]{article}

\usepackage{amsmath, amssymb, amsthm, verbatim,enumerate,bbm,color,mathtools,enumitem} 

\usepackage{mathrsfs}

\usepackage[backref,colorlinks,citecolor=blue,bookmarks=true]{hyperref}

\author{Asaf Cohen Antonir\thanks{School of Mathematical Sciences, Tel Aviv University, Tel Aviv, 6997801, Israel.
Email: asafc1$@$tauex.tau.ac.il} \and Asaf Shapira \thanks{School of Mathematics, Tel Aviv University, Tel Aviv 69978, Israel. Email: asafico@tau.ac.il. Supported in part
by ERC Consolidator Grant 863438 and NSF-BSF Grant 20196.}}

\date{\today}
\allowdisplaybreaks

\theoremstyle{plain}
\newtheorem{theorem}{Theorem}[section]
\newtheorem{lemma}[theorem]{Lemma}
\newtheorem{claim}[theorem]{Claim}

\newtheorem{definition}[theorem]{Definition}

\makeatletter
\def\moverlay{\mathpalette\mov@rlay}
\def\mov@rlay#1#2{\leavevmode\vtop{%
   \baselineskip\z@skip \lineskiplimit-\maxdimen
   \ialign{\hfil$\m@th#1##$\hfil\cr#2\crcr}}}
\newcommand{\charfusion}[3][\mathord]{
    #1{\ifx#1\mathop\vphantom{#2}\fi
        \mathpalette\mov@rlay{#2\cr#3}
      }
    \ifx#1\mathop\expandafter\displaylimits\fi}
\makeatother

\makeatletter
\renewenvironment{proof}[1][\proofname]
{\par\pushQED{\qed}
	\normalfont\topsep6\p@\@plus6\p@\relax\trivlist
	\item[\hskip\labelsep\bfseries#1\@addpunct{.}]
	\ignorespaces}
{\popQED\endtrivlist\@endpefalse}
\makeatother

\DeclareMathOperator{\Opt}{Opt}

\RequirePackage{color}\definecolor{RED}{rgb}{1,0,0}\definecolor{BLUE}{rgb}{0,0,1} 

\title{Bounding the number of odd paths in planar graphs via convex optimization}
\date{}
\begin{document}

\maketitle

    \begin{abstract}
    Let $N_{\mathcal{P}}(n,H)$ denote the maximum number of copies of $H$ in an $n$ vertex planar graph.
    The problem of bounding this function for various graphs $H$ has been extensively studied since the 70's.
    A special case that received a lot of attention recently is when $H$ is the path on $2m+1$ vertices, denoted $P_{2m+1}$.
    Our main result in this paper is that
    $$
    N_{\mathcal{P}}(n,P_{2m+1})=O(m^{-m}n^{m+1})\;.
    $$
    This improves upon the previously best known bound by a factor $e^{m}$,
    which is best possible up to the hidden constant, and makes a significant step towards resolving conjectures
    of Gosh et al. and of Cox and Martin. The proof uses graph theoretic arguments together with (simple) arguments
    from the theory of convex optimization.
    \end{abstract}

    \section{Introduction}

    In this paper we study the following extremal problem: given a fixed graph $H$, what is the maximum number of copies of $H$
    that can be found in an $n$ vertex planar graph? We denote this maximum by $N_{\mathcal{P}}(n,H)$.
    The investigation of this problem was initiated by Hakimi and Schmeichel \cite{HakSch1979} in the 70's.
    They considered the case when $H$ is a cycle of length $m$, denoted $C_m$.
    They determined $N_{\mathcal{P}}(n,C_3)$ and $N_{\mathcal{P}}(n,C_4)$ exactly, and for general $m\geq 3$ proved that $N_{\mathcal{P}}(n,C_{m})=\Theta(n^{\lfloor{m}/{2}\rfloor})$.
    Following this result, Alon and Caro \cite{AloCar1984} determined $N_{\mathcal{P}}(n,K_{2,m})$ exactly for all $m$, where $K_{2,m}$ is the complete $2$-by-$m$ bipartite graph. In a series of works \cite{Epp1993,GyoPauSalTomZam2021,HuyWoo2022,Wor1986}, which culminated with a recent paper of Huynh, Joret and Wood \cite{HuyJorWoo2020}, the asymptotic value of $N_{\mathcal{P}}(n,H)$ was determined up to a constant factor\footnote{This line of research was also generalized to other families of sparse host graphs, e.g.\ graphs that are embeddable in a surface of genus $g$, $d$-degenerate graphs, and more. In fact, the main result of \cite{HuyJorWoo2020} also determines (up to constant factors) the maximum number of copies of a given graph in an $n$ vertex graph which is embeddable in a surface of genus $g$. A recent far reaching generalization of \cite{HuyJorWoo2020} can be found in Liu \cite{Liu2021} where the order of magnitude of the maximum number of copies of a given graph in a `nowhere dense' graph was computed up to constant factors.} (depending on $H$) for every fixed $H$.

    The next natural question following the result of \cite{HuyJorWoo2020} is to determine the asymptotic growth of $N_{\mathcal{P}}(n,H)$ up to\footnote{We use the standard notation $o(1)$ to denote a quantity tending to $0$ when $n$ tends to infinity and $H$ is fixed. Similarly, when we write $o(n^k)$ we mean $o(1)\cdot n^k$. } $1+o(1)$, or more ambitiously, to determine its exact value. This line of research was initiated by Gy\H{o}ri, Paulos, Salia, Tompkins and Zamora \cite{GyoPauSalTomZam2019,GyoPauSalTomZam2021}, who showed that for large enough $n$ we have $N_{\mathcal{P}}(n,P_4)=7n^{2}-32n+27$ and  $N_{\mathcal{P}}(n,C_5)=2n^2-10n+12$, where $P_m$ denotes the path with $m$ vertices (and $m-1$ edges).
    We note that the result of Alon and Caro \cite{AloCar1984} implies that $N_{\mathcal{P}}(n,K_{1,2})=N_{\mathcal{P}}(n,P_3)= n^2+3n-16$.
    Addressing the problem of finding the asymptotic value of $N_{\mathcal{P}}(n,P_{m})$ up to $1+o(1)$, Ghosh, Gy\H{o}ri, Martin, Paulos, Salia, Xiao and Zamora \cite{GhoGyoMarPauSalXiaZam2021}, showed that $N_{\mathcal{P}}(n,P_{5})=(1+o(1))n^3$. They also raised the following conjecture\footnote{They also conjectured that the second order term is $O(n^{m})$.}
    regarding the asymptotic value of $N_{\mathcal{P}}(n,P_{2m+1})$ for arbitrary $m \geq 2$:
    \begin{equation}\label{eq-Conjecture}
       N_{\mathcal{P}}(n,P_{2m+1})=(4m^{-m}+o(1))n^{m+1}\;.
    \end{equation}
    We note that the lower bound in \eqref{eq-Conjecture} is easy. Indeed, start with a cycle of length $2m$, and then replace every second vertex with an independent set consisting of $(n-m)/m$ vertices, each with the same neighborhood as the original vertex it replaced.

    In a very recent paper, Cox and Martin \cite{CoxMar2021_1} introduced an analytic approach for proving \eqref{eq-Conjecture}.
    They showed that
    \begin{equation}\label{eqcm}
    N_{\mathcal{P}}(n,P_{2m+1})\leq (\rho(m)/2+o(1))n^{m+1}\;,
    \end{equation}
    where $\rho(m)$ is the solution to a certain convex optimization problem, which we define precisely in Section \ref{section3}.
    They further conjectured that
    \begin{equation}\label{CoxMartinConj}
        \rho(m)\leq 8m^{-m}\;,
    \end{equation}
    which, if true, implies \eqref{eq-Conjecture}. In the same paper, they verified their conjecture for $m=3$ by showing that $\rho(3)=8/27$, which confirms \eqref{eq-Conjecture} for $m=3$.
    Using the same approach they also improved the known asymptotic value of $N_{\mathcal{P}}(n,P_{2m+1})$ by showing that
    \[
        N_{\mathcal{P}}(n,P_{2m+1})\leq \left(\frac{1}{2\cdot(m-1)!}+o(1)\right)n^{m+1}\;.
    \]
    Note that this bound is roughly $e^{m}$ larger than the one conjectured in (\ref{eq-Conjecture}).

    Our main result in this paper, Theorem \ref{thm:number of path in planar graphs} below, makes a significant step towards the resolution of the Cox--Martin and Gosh et al.\ conjectures, by establishing \eqref{CoxMartinConj} up to an absolute constant.

    \begin{theorem}\label{thm:number of path in planar graphs}
    There is an absolute constant $C$ so that for every fixed $m\geq 2$ and large enough $n$, we have
    \[
        N_{\mathcal{P}}(n,P_{2m+1})\leq Cm^{-m}n^{m+1}\;.
    \]
    \end{theorem}

    As noted after \eqref{eq-Conjecture}, the above bound is best possible up to the value of $C$. Furthermore, as can be seen in the proof of Theorem \ref{thm:number of path in planar graphs}, the constant $C$ we obtain is $10^4$ (which can certainly be improved).

    \subsection{Related work and paper overview}

    In addition to studying $N_{\mathcal{P}}(n,P_{2m+1})$, Cox and Martin \cite{CoxMar2021_1} also introduced an analytic method for bounding the maximum number of even cycles in planar graphs.
    Similar to the case of odd paths discussed above, they showed that $N_{\mathcal{P}}(n,C_{2m})\leq (\beta(C_m)+o(1)) n^{m}$, where $\beta(C_m)$ is an optimization problem, similar to the one we study in Section \ref{section2}.
    They conjectured that $\beta(C_m)=m^{-m}$, a bound which implies by their framework that $N_{\mathcal{P}}(n,C_{2m})\leq (1+o(1)) (n/m)^{m}$.
    Observe that the example we mentioned after \eqref{eq-Conjecture} shows that this bound is best possible.
    Towards their conjecture, Cox and Martin \cite{CoxMar2021_1} proved that $\beta(C_m)\leq 1/m!$.
    Using the ideas in this paper, one can significantly improve this bound. In particular, using Lemma \ref{lem:number of P_m in an n vertex graph} in Section \ref{section2}, it is not hard to show that for some absolute constant $C$, we have
    \begin{equation}\label{eqcycles}
     \beta(C_m)\leq Cm^{-m}\;.
    \end{equation}
    In an independent work, Lv, Gy\H{o}ri, He, Salia, Tompkins and Zhu \cite{LvGyoHeSalTomZhu2022} confirmed the conjecture of Cox and Martin by showing that one can in fact obtain $C=1$ in (\ref{eqcycles}). We thus do not include the proof of (\ref{eqcycles}).

    We should point that the reason why studying $N_{\mathcal{P}}(n,P_{2m+1})$ appears to be much harder than $N_{\mathcal{P}}(n,C_{2m})$ is that
    as opposed to $\beta(C_m)$, which is an optimization problem involving a single graph, $\rho(m)$ is an optimization problem which involves several multigraphs. To overcome this difficulty
    we first study in Section \ref{section2} an optimization problem, denoted $\beta(P_m)$,
    which is the analogue of $\beta(C_m)$ for the setting of $P_m$. The main advantage of first studying $\beta(P_m)$ is that it allows us
    to employ a weight shifting argument, which does not seem to be applicable to $\rho(m)$.
    Our main result in that section is a nearly tight bound for $\beta(P_{m})$.
    However, as opposed to the case of $N_{\mathcal{P}}(n,C_{2m})$, a bound for $\beta(P_{m})$ does not immediately translate into a bound for $N_{\mathcal{P}}(n,P_{2m+1})$. Hence, in Section \ref{section3} we present the main novel part of this paper, showing how one can transfer any bound for $\beta(P_{m})$ into a bound for $\rho(m)$, thus proving Theorem \ref{thm:number of path in planar graphs}. To this end we use simple arguments from the theory of convex optimization, which allow us to exploit the fact that $\rho(m)$ is a low degree polynomial.

    The key lemmas leading to the proof of Theorem \ref{thm:number of path in planar graphs} are Lemmas \ref{thm:main theorem} and \ref{prop:rho to beta} for which we obtain bounds that are optimal up to constant factors. Moreover, if one can improve these bounds to the optimal conjectured ones, then this will give the conjectured inequality (\ref{eq-Conjecture}). We believe that with more care, it is possible to use the ideas in this paper to improve the bound for $\beta(P_m)$ in Lemma \ref{thm:main theorem} to the conjectured one.
    In contrast, because of the complex structure of $\rho(m)$, it seems that in order to improve the bound in Lemma \ref{prop:rho to beta} to the conjectured bound, a new idea is needed.

    \section{A variant of $\rho(m)$}
    \label{section2}
    Our goal in this section is to prove Lemma \ref{thm:main theorem} regarding the optimization problem $\beta(P_m)$.
    This lemma will be used in the next section in the proof of Theorem \ref{thm:number of path in planar graphs}.
    The proof of Lemma \ref{thm:main theorem} will employ a subtle weight shifting argument.
    We first recall several definitions from \cite{CoxMar2021_1}. In what follows we write $[n]$ to denote the set $\{1,\ldots,n\}$ and $K_n$ to denote the complete graph on $[n]$.

    \begin{definition}
    Let $n>0$ be an integer, and let $\mu$ be a probability measure on the edges of $K_{n}$.

    \begin{enumerate}
        \item For any $x\in [n]$ we define the weighted degree of $x$ to be
        \[
            \bar{\mu}(x)=\sum_{y\in [n]\setminus\{x\}}\mu(x,y)\;.
        \]
        \item For any subgraph $H\subseteq K_n$ we define the weight of $H$ to be
        \[
            \mu(H)=\prod_{e\in E(H)}  \mu(e)\;.
        \]
        \item For any graph $H$ with no isolated vertices, define
        \[
            \beta(\mu;H)=\sum_{H'\in \mathbf{C}(H,n)}\mu(H')\;,
        \]
        where $\mathbf{C}(H,n)$ is the set of all (non-induced and unlabeled) copies of $H$ in $K_n$.
        Further, we define
        \[
            \beta(H)=\sup_{\mu} \beta(\mu;H)\;,
        \]
        where the supremum is taken over all $n'$ and all probability measures $\mu$ on the edges of $K_{n'}$.
    \end{enumerate}
    \end{definition}

    Intuitively, the function $\beta(\mu;H)$ is the probability of hitting a (non-induced and unlabeled) copy of $H$ if $|E(H)|$ independent edges were chosen according to $\mu$.

    \begin{lemma}\label{thm:main theorem}
        For any integer $m\geq 2$ we have
        \begin{equation*}
            \beta(P_{m})\leq \frac{2e^2}{m^{m-2}}\;.
        \end{equation*}
    \end{lemma}

    We remark that this lemma is optimal up to the constant factor $2e^2$. To see this, consider the uniform distribution over the edges of $C_{m}$, which shows that $\beta(P_{m}) \geq 1/m^{m-2}$. It seems reasonable to conjecture that $\beta(P_{m}) = 1/m^{m-2}$.

    The key step in the proof of Lemma \ref{thm:main theorem} is Lemma \ref{lem:number of P_m in an n vertex graph} below. To state this lemma, we first need the following definitions.

    \begin{definition}
    For every $k,\ell \geq 0$ we define $P_{(k,\ell)}$ to be a disjoint union of $P_{k+1}$ and $P_{\ell+1}$.
    \end{definition}

    From now on, we will not only deal with probability measures but also with bounded measures. Therefore, we will frequently write \emph{measure} to denote a bounded measure. Moreover, for a measure $\mu$ we will denote its total mass by $w(\mu)$.

    \begin{definition}\label{defbetastar}
        Suppose $\mu$ is a measure on the edges of $K_n$ and $s,t \geq 0$. Define
        \[
            \beta^*(\mu;P_{(s,t)}) = \sum _{P\in\mathbf{C}^*(P_{(s,t)},n)}\mu(P)\;,
        \]
        where $\mathbf{C}^*(P_{(s,t)},n)$ is the set of copies of $P_{(s,t)}$ in $K_{n}$ where the path of length $s$ starts with the vertex $n$, and the path of length $t$ starts with the vertex $1$. Further, for every $w>0$ we define
        \[
            \beta_{w,n}^*(P_{(s,t)})=\sup_{\mu} \beta^*(\mu;P_{(s,t)})\;,
        \]
        where the supremum is taken over all measures $\mu$ on the edges of $K_{n}$ with $w(\mu)=w$.
    \end{definition}

    We remark that for any measure $\mu$ on the edges of $K_n$, we have $\beta^*(\mu;P_{(0,0)})=1$. This is because $\mathbf{C}^*(P_{(0,0)},n)$ consist of a single graph, the independent set $I_2=\{1,n\}$, and because $\mu(I_2)=1$. This clearly implies that $\beta_{w,n}^*(P_{(0,0)})=1$ for every $w$ and $n$.

    \begin{lemma}\label{lem:number of P_m in an n vertex graph}
        For every $0\leq \ell \leq m \leq n$ we have
        \[
            \beta_{1,n}^*(P_{(\ell,m-\ell)})\leq \frac{1}{m^{m}}\;.
        \]
    \end{lemma}

    \begin{claim}\label{claim_double_path}
        Suppose that $t$ is a non-negative integer, $s,n$ are positive integers, and $w\geq 0$. Then, there exists a measure $\mu$ on the edges of $K_n$ with $w(\mu)=w$, satisfying:
        \begin{enumerate}
            \item $\beta^*(\mu;P_{(s,t)})=\beta^*_{w,n}(P_{(s,t)})$, and
            \item for all $q\neq n-1$ we have $\mu(q,n)=0$.
        \end{enumerate}
    \end{claim}

    \begin{proof}
        The main idea in the proof is the introduction of the notion of a $w$-useful measure.
        We say that a measure $\mu$ on the edges of $K_n$ with $w(\mu)=w$ is \emph{$w$-optimal} if
        \[
            \beta^*(\mu;P_{(s,t)})=\beta_{w,n}^*(P_{(s,t)})\;.
        \]
        We further say that $\mu$ is \emph{$w$-useful} if $\mu$ is $w$-optimal and
        \[
            \max_{k\in[n-1]} \mu(n,k) = \sup_{\eta,k} \eta(n,k)\;,
        \]
        where the supremum is taken over all $k\in [n-1]$ and all measures $\eta$ which are $w$-optimal. Let us see why such a $w$-useful measure exists.
        Note that there is a natural bijection between measures $\mu$ with $w(\mu)=w$, and vectors in the simplex $\Delta=\{x\in \mathbb{R}^{\binom{n}{2}}:x_i \geq 0 ~\mbox{and}~\sum_{i=1}^{\binom{n}{2}} x_i=w\}$. Thus, to show that a $w$-useful measure exists we think of $\mu$ as a vector in $\Delta$.
        Recalling that
        \[
            \beta^*(\mu;P_{(s,t)}) = \sum _{P\in\mathbf{C}^*(P_{(s,t)},n)}\mu(P)= \sum_{P\in\mathbf{C}^*(P_{(s,t)},n)}\prod_{e\in E(P)}  \mu(e)\;,
        \]
        we see that $\beta^*(\mu;P_{(s,t)})$ is an $\binom{n}{2}$-variate polynomial, with variables $\mu(e)$ for all $e\in E(K_{n})$. Under these notations, $w$-optimal measures are maximal points of the polynomial $\beta^*(\mu;P_{(s,t)})$ in $\Delta$. Since $\Delta$ is compact and $\beta^*(\mu;P_{(s,t)})$ is continuous, we deduce that $O_w$, the set of all $w$-optimal measures, is non-empty. Moreover, $O_w$ is a compact set, since it is closed (as the preimage of a closed set under the continuous function $\beta^*(\mu;P_{(s,t)})$) and bounded (as it is contained in $\Delta$).
        Setting $f(\mu)=\max_{k\in[n-1]}\mu(n,k)$, we find that $\mu$ is a $w$-useful measure if and only if it is a maximal point of $f$ within $O_w$. Since $O_w$ is compact and $f$ is continuous, a $w$-useful measure exists.

        We now prove that the existence of $w$-useful measures implies the claim. Indeed, let $\mu$ be a $w$-useful measure. Assume with out loss of generality\footnote{If this is not the case, we can permute the vertices and end up with such measure.} that $\mu(n-1,n)\geq 0$ is maximal among all $\mu(k,n)$. We claim that $\mu$ is as required. The first condition follows immediately from the fact that any $w$-useful measure is also $w$-optimal.
        Assume towards contradiction that the second condition fails, that is, that there exists a $q\neq n-1$ with $\mu(q,n)>0$. We will now show that there is a measure $\mu'$ satisfying $w(\mu')=w$ which will either
        contradict the fact that $\mu$ is $w$-optimal or the fact that it is $w$-useful.

        We define $\mu'$ as follows:
        We first set $\mu'(e)=\mu(e)$ for every edge other than the two edges $\{n-1,n\}$ and $\{q,n\}$.
        Define $W_q$ to be the weight (under $\mu$) of all copies of $P_{(s-1,t)}$, not containing $n$, such that the path of length $s-1$ starts with $q$, and the path of length $t$ starts with $1$. Define $W_{n-1}$ analogously. Then, we define
        \[
            \mu'(n,n-1)=\begin{cases}
                        \mu(n-1,n)+\mu(q,n) & \text{if } W_{n-1}\geq W_q\;,\\
                        0 & \text{else}\;,
                    \end{cases}
        \]
        and
        \[
            ~~~~~~\mu'(q,n)=\begin{cases}
                        0 & \text{if } W_{n-1}\geq W_q\;,\\
                        \mu(q,n)+\mu(n-1,n) & \text{else\;.}
                    \end{cases}
        \]
        To see that we indeed get a contradiction, assume first that $W_{n-1}\geq W_q$.
Since a copy of $P_{(s,t)}$ in $C^{*}(P_{(s,t)},n)$ uses at most one of the edges $\{n-1,n\}$ and $\{q,n\}$, decreasing the value of
$\{q,n\}$ by some $\varepsilon$ while increasing that of $\{n-1,n\}$ by the same $\varepsilon$ increases the
total weight of copies of $P_{(s,t)}$ by $\varepsilon(W_{n-1}-W_q)$. We thus infer that
\begin{align*}
            \beta^*(\mu';P_{(s,t)}) = \beta^*(\mu;P_{(s,t)})+\mu(q,n)(W_{n-1}-W_q) \geq \beta^*(\mu;P_{(s,t)})\;.
\end{align*}
Since $\mu'(n,n-1)>\mu(n,n-1)$ we see that $\mu'$ witnesses the fact that $\mu$ is not $w$-useful.
        If on the other hand $W_{q}>W_{n-1}$, then
        \begin{align*}
            \beta^*(\mu';P_{(s,t)}) = \beta^*(\mu;P_{(s,t)})+\mu(n-1,n)(W_{q}-W_{n-1})> \beta^*(\mu;P_{(s,t)})=\beta_{w,n}^*(P_{(s,t)})\;,
        \end{align*}
        so $\mu'$ witnesses the fact that $\mu$ is not $w$-optimal.
    \end{proof}

    \begin{claim}\label{claim_induction}
    Suppose $s,t$ are non-negative integers, $n$ is a positive integer, and $w\geq 0$. Then, there are $w_1,\ldots ,w_s\geq 0$ such that $\sum_{i=1}^{s}w_i\leq w$ and such that
    \[
        \beta_{w,n}^*(P_{s,t})\leq \beta_{w',n-s}^*(P_{0,t}) \cdot \prod_{i=1}^{s}w_i \;,
    \]
    where $w'=w-\sum_{i=1}^{s}w_i$.
    \end{claim}

    \begin{proof}
    First, if $s+t+1\geq n$ then the claim is trivial, as $\mathbf{C}^*(P_{(s,t)},n)=\emptyset$. So we assume for the rest of the proof that $s+t+2\leq n$.
    Let $\mu_0$ be a measure on the edges of $K_n$ as guaranteed by Claim \ref{claim_double_path}. Since $\beta_{w,n}^*(P_{(s,t)})=\beta^*(\mu_0;P_{(s,t)})$, it is enough to prove that there are $w_1,\ldots,w_s\geq 0$ such that $\sum_{i=1}^{s}w_i\leq w$ and
    \begin{equation}\label{eqbeta1}
      \beta^*(\mu_0;P_{(s,t)})\leq \beta_{w',n-s}^*(P_{0,t})\cdot \prod_{i=1}^{s}w_i\;,
    \end{equation}
    where $w'=w-\sum_{i=1}^{s}w_i$. We define inductively a sequence of reals $w_1,\ldots w_k\geq 0$ with $\sum_{i=1}^{k}w_i\leq w$, along with measures $\mu_1,\ldots ,\mu_k$ on the edges of $K_{n-1},\ldots ,K_{n-k}$, respectively, such that the following holds for all $1 \leq j \leq s$, where we set $w'_{j}=w-\sum_{i=1}^{j}w_i$:
    \begin{enumerate}[label=(\roman*)]
        \item\label{1} $w(\mu_j)=w'_{j}$,
        \item\label{2} $w_{j}=\mu_{j-1}(n-j+1,n-j)$,
        \item\label{3} for all $t\in[n-j-2]$ we have $\mu_j(n-j,t)=0$,
        \item\label{4} $\beta^*(\mu_j;P_{(s-j,t)})=\beta_{w'_{j},n-j}^*(P_{(s-j,t)})$, and
        \item\label{5} $\beta^*(\mu_{j-1};P_{(s-j+1,t)})\leq w_j \cdot \beta^*(\mu_{j};P_{(s-j,t)})$.
    \end{enumerate}
    Indeed, assuming $w_1,\ldots, w_j$ and $\mu_1,\ldots,\mu_j$ have already been chosen, we now choose $w_{j+1}$ and $\mu_{j+1}$.
    We first set $w_{j+1}=\mu_j(n-j,n-j-1)\geq 0$ so that the second condition holds.
    Further, set $\mu'_{j+1}=\mu_{j}|_{n-j-1}$, the restriction of $\mu_{j}$ to the edges of $K_{n-j-1}$.
    Observe that by the induction hypothesis on $\mu_{j}$, we have $\mu_{j}(n-j,t)=0$ for all $t\neq n-j-1$.
    Hence
    \[
        w(\mu'_{j+1})=w(\mu_{j})-\sum_k\mu_{j}(n-j,k)=w(\mu_{j})-w_{j+1}=w'_{j+1}\;,
    \]
    and
    \begin{align}\label{eq1-lemma}
        \beta^*(\mu_j;P_{(s-j,t)})= w_{j+1}\cdot \beta^*(\mu'_{j+1};P_{(s-j-1,t)})\leq w_{j+1}\cdot \beta_{w'_{j+1},n-j-1}^*(P_{(s-j-1,t)})\;.
    \end{align}
    Let $\mu_{j+1}$ be the measure given by Claim \ref{claim_double_path} applied with $P_{(s-j-1,t)}$ and total mass $w'_{j+1}$.
    We claim that $\mu_{j+1}$ satisfies the inductive properties. The fact that it satisfies the first condition is immediate from its definition.
    To see that $\mu_{j+1}$ satisfies the last three conditions, note that by Claim \ref{claim_double_path} the measure $\mu_{j+1}$ satisfies
    \begin{equation}\label{eq2-lemma}
        \beta^*(\mu_{j+1};P_{(s-j-1,t)})=\beta_{w'_{j+1},n-j-1}^*(P_{(s-j-1,t)})\;,
    \end{equation}
    and $\mu_{j+1}(n-j-1,t)=0 $ for all $t\neq n-j-2$. Finally, combining \eqref{eq1-lemma} and \eqref{eq2-lemma} we obtain
    \[
        \beta^*(\mu_j;P_{(s-j,t)})\leq w_{j+1}\cdot\beta^*(\mu_{j+1};P_{(s-j-1,t)})\;,
    \]
    thus verifying the last three properties.
    Repeatedly applying property \ref{5} we deduce that
    $$
    \beta^*(\mu_0;P_{(s,t)}) \leq \beta^*(\mu_s;P_{(0,t)})\cdot \prod_{i=1}^{s}w_i\;.
    $$
Since $\beta^*(\mu_s;P_{(0,t)})= \beta^*_{w'_s,n-s}(P_{(0,t)})=\beta^*_{w',n-s}(P_{(0,t)})$ (by property \ref{4} and the definition of $w'$) we have thus proved (\ref{eqbeta1}) and the proof is complete.
    \end{proof}

    We now use Claim \ref{claim_induction} to prove Lemma \ref{lem:number of P_m in an n vertex graph}.

    \begin{proof}[Proof of Lemma \ref{lem:number of P_m in an n vertex graph}]
        Claim \ref{claim_induction} applied with $s=\ell,t=m-\ell$ and with $w=1$ asserts that there are $w_1,\ldots ,w_\ell\geq 0$ such that $\sum_{i=1}^{\ell}w_i\leq 1$ and such that
        \begin{equation}\label{eq1-main-lemma}
            \beta_{1,n}^*(P_{\ell,m-\ell})\leq \beta_{w',n-\ell}^*(P_{0,m-\ell}) \cdot \prod_{i=1}^{\ell}w_i \;,
        \end{equation}
        where $w'=1-\sum_{i=1}^{\ell}w_i$.
        Clearly, for all integers $s,t,k$ and $w\geq 0$ we have $\beta_{w,k}^*(P_{(s,t)})=\beta_{w,k}^*(P_{(t,s)})$.
        Hence, using Claim \ref{claim_induction} with $s=m-\ell,t=0$ and with $w=w'$, we obtain a sequence $w_{\ell+1},\ldots,w_{m}$ of non-negative reals, such that $\sum_{i=\ell+1}^m w_i\leq w'$ and such that
        \begin{equation}\label{eq2-main-lemma}
             \beta_{w',n-\ell}^*(P_{0,m-\ell})=\beta_{w',n-\ell}^*(P_{m-\ell,0})\leq \beta_{w'',n-m}^*(P_{(0,0)})\cdot \prod_{i=\ell+1}^{m}w_i =\prod_{i=\ell+1}^m w_i\;,
        \end{equation}
        where $w''=w'-\sum_{i=\ell+1}^m w_i$, and we used the fact that $\beta_{w'',n-m}^*(P_{(0,0)})=1$ (see the remark after Definition \ref{defbetastar}).
        Combining \eqref{eq1-main-lemma} and \eqref{eq2-main-lemma}, we infer that there are $w_1,\ldots,w_m\geq 0$ with $\sum_{i=1}^{m}w_i\leq 1$ such that
        \[
            \beta_{1,n}^*(P_{\ell,m-\ell})\leq \prod_{i=1}^{m}w_i\leq \left(\frac{\sum_{i=1}^{m}w_i}{m}\right)^m\leq \frac{1}{m^m}\;,
        \]
        where the second inequality is the AM-GM inequality, and the last inequality follows from the properties of the sequence $w_1,\ldots ,w_m$.
        \end{proof}

    To deduce Lemma \ref{thm:main theorem} from the above claims, we recall a definition and a lemma from Cox and Martin \cite{CoxMar2021_1} which we specialize here to the case of $P_m$.

    \begin{definition}
        For an integer $n$, we denote by $\Opt(n;H)$ the set of all probability measures $\mu$ on the edges of $K_{n}$ satisfying
        \[
            \beta(\mu; P_m) = \sup_{\eta} \beta(\eta;P_m)\;,
        \]
        where the supremum is taken over all probability measures $\eta$ on the edges of $K_{n}$.
    \end{definition}

    \begin{lemma}[Lemma 4.5 in \cite{CoxMar2021_1}]\label{lem:inequalities regarding the masses}
    For every $n\geq m\geq 2$ and $\mu \in \Opt(n; P_m)$, we have the following for all $x \in [n]$
    \[
        \bar{\mu}(x)\cdot (m-1)\cdot \beta(\mu;P_m)=\sum_{\substack{P\in \mathbf{C}(P_m,n)\\ V(P)\ni x}}\deg_{P}(x)\mu (P)\;.
    \]
    \end{lemma}

    \begin{proof}[Proof of Lemma \ref{thm:main theorem}]
        Suppose $n \geq m$ and take any $\mu\in \Opt(n;P_m)$. We will next show that $\beta(\mu;P_m)\leq \frac{20}{m^{m-2}}$ thus completing the proof.
        Let $x\in [n]$ be such that $\bar{\mu}(x)\neq 0$. By Lemma \ref{lem:inequalities regarding the masses} we have
        \begin{align}\label{eq3}
            \bar{\mu}(x)\cdot (m-1)\cdot \beta(\mu;P_m)\leq 2\sum_{\substack{P\in \mathbf{C}(P_m,n)\\ V(P)\ni x}}\mu (P)\;.
        \end{align}
        Given distinct $s,t\in[n]$ and $0 \leq \ell \leq m-2$ we define $\mathbf{C}^*(s,t,\ell)$ to be the set of all copies of $P_{(\ell,m-\ell-2)}$ in $K_n$, where the path of length $\ell$ starts with $s$ and the path of length $m-\ell-2$ starts with $t$.
        We have
        \begin{align}
               \sum_{\substack{P\in \mathbf{C}(P_m,n)\\ V(P)\ni x}}\mu (P)&= \sum_{y\in [n]\setminus \{x\}}\mu(x,y)\sum_{\ell=0}^{m-2}\sum_{\substack{P \in \mathbf{C}^*(x,y,\ell)}}\mu (P)\nonumber\\
               &\leq \sum_{y\in [n]\setminus \{x\}}\mu(x,y)\sum_{\ell=0}^{m-2}\beta^*_{1,n}(P_{(\ell,m-\ell-2)})\nonumber\\
               &\leq \frac{m-1}{(m-2)^{m-2}}\sum_{y\in [n]\setminus \{x\}}\mu(x,y)=\frac{\bar{\mu}(x)(m-1)}{(m-2)^{m-2}}\label{eq4}\;,
        \end{align}
        where the second inequality holds by the definition\footnote{We rely on the fact that although $\beta^*_{w,n}(P_{(s,t)})$ was defined with respect to paths starting at vertices $1$ and $n$, we could have chosen any pair of vertices in $[n]$ (in the above proof we use $x,y$). } of $\beta^*_{1,n}(P_{(s,t)})$, and the third inequality holds by Lemma \ref{lem:number of P_m in an n vertex graph}.
        Recalling that $\bar{\mu}(x)>0$ and combining \eqref{eq3} and \eqref{eq4} we infer that
        \[
            \beta(\mu;P_m)\leq \frac{2}{(m-2)^{m-2}}\leq \frac{2e^2}{m^{m-2}}\;.\qedhere
        \]
    \end{proof}

    \section{Proving the main result}
    \label{section3}

    We start this section with stating the optimization problem of Cox and Martin \cite{CoxMar2021_1}.

    \begin{definition}
        Let $n$ be an integer and let $\mu$ be a probability measure on the edges of $K_{n}$. Then, for any integer $m\geq 2$, letting $(n)_m$ to be the set of all ordered $m$-tuples of distinct elements from $[n]$, define
        \[
            \rho (\mu;m)=\sum_{x\in (n)_m}\bar{\mu}(x_1)\left(\prod_{i=1}^{m-1}\mu (x_i,x_{i+1})\right)\bar{\mu}(x_m)\;.
        \]
        Furthermore, define
        \[
            \rho_n(m)=\sup_{\mu}\rho(\mu;m)\quad \text{and}\quad \rho(m)=\sup_{n\in \mathbb{N}}\rho_n(m)\;,
        \]
        where the supremum in the definition of $\rho_n(m)$ is taken over all probability measures $\mu$ on the edges of $K_{n}$.
    \end{definition}

Note that if we expand the products in the definition of $\rho(m)$ we see that
$\rho(m)$ is very similar to $\beta(P_{m+2})$. The crucial difference is that in $\rho(m)$ we count the total weight of walks of a very special structure. These walks are formed by first choosing distinct $x_2,\ldots ,x_{m+1}$ to be a copy of $P_{m}$, and then choosing \emph{arbitrary} $x_1\neq x_2$ and $x_{m+2}\neq x_{m+1}$ (so we allow $x_1=x_{m+1}$ and/or $x_1,x_{m+2}\in \{x_2,\ldots ,x_{m+1}\}$). For example, a walk of this type might be $(1,2,1,2)$ or $(1,2,1,3,1)$.

Our main task in this section is to prove the following lemma.

    \begin{lemma}\label{prop:rho to beta}
    For all integers $n\geq m\geq 2$ we have
    \[
        \rho_n(m)\leq \frac{1152}{m^2}\cdot \beta(P_{m})\;.
    \]
    \end{lemma}

    The constant $1152$ in the above lemma is clearly not optimal. We did not make any attempt to improve it, as it seems that a new idea is required to obtain the optimal one.
    A simple lower bound for $\rho_n(m)$ is $8/m^{m}$, which is achieved by the uniform distribution on the edges of $C_{m}$.
    As we mentioned in the previous section, it seems reasonable to conjecture that $\beta(P_{m})=1/m^{m-2}$.
    Therefore, a natural conjecture is that in Lemma \ref{prop:rho to beta} the optimal constant is $8$.

    Let us first deduce Theorem \ref{thm:number of path in planar graphs} from Lemmas \ref{thm:main theorem} and \ref{prop:rho to beta}.

    \begin{proof}[Proof of Theorem \ref{thm:number of path in planar graphs}]
        Lemma 2.3 in Cox and Martin \cite{CoxMar2021_1} asserts that for all $m\geq 2$ we have
        \[
            N_{\mathcal{P}}(n,P_{2m+1})\leq (\rho(m)/2+o(1)) n^{m+1}\;.
        \]
        Furthermore, since Lemma \ref{prop:rho to beta} holds for all $n$, we deduce that $\rho(m)\leq \frac{10^3}{m^2}\cdot\beta(P_m)$.
        Together with Lemma \ref{thm:main theorem}, this gives Theorem \ref{thm:number of path in planar graphs} as then
        \begin{align*}
            N_{\mathcal{P}}(n,P_{2m+1})&\leq (\rho(m)/2+o(1)) n^{m+1}\leq (576\beta(P_m)m^{-2}+o(1))  n^{m+1}\\
            &\leq  (10^4m^{-m}+o(1)) n^{m+1}\;.\qedhere
        \end{align*}
    \end{proof}

    Before proving Lemma \ref{prop:rho to beta}, let us recall a special case of the Karush--Kuhn--Tucker (KKT) conditions (see Corollaries 9.6 and 9.10 in \cite{Gul2010}).

    \begin{theorem}[Special case of the KKT conditions]\label{thm:KKT}
    Let $f\colon \mathbb{R}^n\to \mathbb{R}$ be a continuously differentiable function, and consider the optimization problem
    \begin{align*}
        \max_{x\in \Delta} f(x)\;, \text{ where } \Delta=\left\{x:\sum_{i=1}^{n} x_i=1\text{ and } x_1,\ldots ,x_n\geq 0\right\}\;.
    \end{align*}
    If $\mathbf{x}^*$ achieves this maximum, then there is some $\lambda\in \mathbb {R}$ such that, for each $i\in [n]$, either
    \[
        \mathbf{x}^*_i=0,\quad\text{or}\quad \frac{\partial f}{\partial x_i}(\mathbf{x}^*)=\lambda\;.
    \]
    \end{theorem}

    \begin{proof}[Proof of Lemma \ref{prop:rho to beta}.]
        Let $\mathbf{P^*}$ be the set of walks $(x_1,x_2,\ldots,x_{m+2})$ on $[n]$ constructed as follows: first, choose $(x_2,x_3,\ldots,x_{m+1})$ to be a path (i.e.\ a non-induced and \emph{labeled} copy of $P_{m}$), and then complete the walk by choosing an arbitrary $x_1\neq x_2$ and an arbitrary $x_{m+2}\neq x_{m+1}$.
        Further, for any $i\neq j\in [n]$ we let $\mathbf{P^*}(\{i,j\})$ be the set of all walks $(x_1,x_2,\ldots,x_{m+2})\in \mathbf{P^*}$ such that there is $k$ with $\{x_{k},x_{k+1}\}=\{i,j\}$.

        Define $f\colon \mathbb{R}^{\binom{[n]}{2}}\to \mathbb{R}$ by
        \[
            f(\mathbf{x})=\sum_{p\in (n)_m}\left(\sum_{p_0\in [n]\setminus\{p_1\}}\mathbf{x}_{p_0,p_1}\right)\left(\prod_{i=1}^{m-1}\mathbf{x}_{p_i,p_{i+1}}\right)\left(\sum_{p_{m+1}\in [n]\setminus \{p_m\}}\mathbf{x}_{p_m,p_{m+1}}\right)\;.
        \]
        Suppose $\mu$ is a probability measure on the edges of $K_n$ with $\rho(\mu;m)=\rho_n(m)$. When viewing $\mu$ as a vector in $\mathbb{R}^{\binom{[n]}{2}}$, we have $f(\mu)=\rho(\mu;m)$, and moreover,
        \[
            f(\mu)=\max_{\mathbf{x}\in \Delta} f(\mathbf{x})\;, \text{ where } \Delta=\left\{\mathbf{x}:\sum_{i=1}^{\binom{n}{2}} \mathbf{x}_i=1\text{ and } \mathbf{x}_1,\ldots ,\mathbf{x}_{\binom{n}{2}}\geq 0\right\}\;.
        \]
    By the maximality of $\mu$ and by Theorem \ref{thm:KKT} (the KKT conditions), there is a non-negative\footnote{As the polynomial has only positive coefficients, $\lambda$ must be non-negative.} real $\lambda$ such that for all $\{i,j\}\in \binom{[n]}{2}$ we have
    \[
        \mu(i,j)=0 \quad \text{or}\quad \frac{\partial f(\mathbf{x})}{\partial \mathbf{x}_{i,j}}(\mu)=\lambda\;.
    \]
    Note that the degree of each term $\mathbf{x}_{i,j}$, in every monomial of $f(\mathbf{x})$ is at most\footnote{The only case where it is $3$ is when $m=2$ and we consider a walk on one edge three times, e.g,\ the walk $(1,2,1,2)$.} $3$. Thus, for every $\{i,j\}\in \binom{[n]}{2}$ we have
    \begin{align}\label{eq1-rho}
       \lambda \cdot \mu(i,j)=\frac{\partial f(\mathbf{x})}{\partial \mathbf{x}_{i,j}}(\mu)\cdot\mu(i,j)\leq 3\sum_{P\in \mathbf{P^*}(\{i,j\})}\mu (P)\;.
    \end{align}
    We also have the following:
    \begin{align}
        \lambda &= \sum_{\{i,j\}\in \binom{[n]}{2}}\lambda \cdot \mu(i,j)\\
        &=\sum_{\{i,j\}\in \binom{[n]}{2}}\frac{\partial f(\mathbf{x})}{\partial \mathbf{x}_{i,j}}(\mu)\cdot \mu(i,j)\nonumber\\
        &\geq \sum_{\{i,j\}\in \binom{[n]}{2}}\sum_{P\in \mathbf{P^*}(\{i,j\})}\mu(P)\nonumber\\
        &=\sum_{P\in \mathbf{P^*}}\mu(P)\sum_{\{i,j\}\in \binom{[n]}{2}}\mathbbm{1} (\{i,j\}\in E(P))\\
        &\geq (m-1)\rho(\mu;m)\label{eq2-rho}\;,
    \end{align}
    where the first equality holds as $\mu$ is a probability measure, the second equality holds by the definition of $\lambda$, and the last inequality holds as there are at least $m-1$ distinct edges in each walk in $\mathbf{P^*}$. Combining \eqref{eq1-rho} and \eqref{eq2-rho} we have the following for all $i\in [n]$:
    \begin{align*}
        (m-1)\cdot \bar{\mu}(i)\cdot \rho(\mu;m) &= \sum_{j\in [n]\setminus\{i\}}(m-1)\mu(i,j)\rho(\mu;m)\\
        &\leq 3 \sum_{j\in [n]\setminus\{i\}} \sum_{P\in \mathbf{P^*}(\{i,j\})}\mu(P)\\
        &=3\sum_{P\in \mathbf{P^*}}\mu (P)\sum_{j\in [n]\setminus\{i\}} \mathbbm{1}(\{i,j\}\in E(P))\\
        &=3\sum_{P\in \mathbf{P^*}}\deg_P(i)\mu (P)\\
        &\leq 12\sum_{P\in \mathbf{P^*}}\mu (P)= 12\cdot \rho(\mu;m)\;.
    \end{align*}
    where the last inequality follows as for every $P\in \mathbf{P^*}$ and $i\in P$ we have\footnote{An example being the walk $(1,2,1,3,1)$.} $\deg_P(i)\leq 4$.
    Dividing both sides by $(m-1)\cdot \rho(\mu;m)$ we obtain that for all $i$ we have $\bar{\mu}(i)\leq \frac{12}{m-1}$. Therefore, as $m\geq 2$ we have
    \begin{align*}
        \rho (\mu;m)&=\sum_{x\in (n)_m}\bar{\mu}(x_1)\left(\prod_{i=1}^{m-1}\mu (x_i,x_{i+1})\right)\bar{\mu}(x_m)\\
        &\leq \frac{144}{(m-1)^2}\sum_{x\in (n)_m}\prod_{i=1}^{m-1}\mu (x_i,x_{i+1})\leq \frac{1152}{m^2}\cdot \beta(P_{m})\;.\qedhere
    \end{align*}

    \end{proof}

    \end{document}